\documentclass[12pt,reqno]{amsart}
\usepackage[top=2cm,bottom=2cm,right=2.5cm,left=2.5cm]{geometry}
\usepackage{amssymb}
\usepackage{amsmath, amsthm, amscd, amsfonts, amssymb, graphicx, color}
\usepackage[bookmarksnumbered, colorlinks, plainpages]{hyperref}

\usepackage{hyperref}

\newtheorem{theorem}{Theorem}[section]

\newtheorem{proposition}[theorem]{Proposition}
\newtheorem{corollary}[theorem]{Corollary}
\theoremstyle{definition}
\newtheorem{definition}[theorem]{Definition}
\newtheorem{example}[theorem]{Example}

\theoremstyle{remark}
\newtheorem{remark}[theorem]{Remark}

\numberwithin{equation}{section}

\begin{document}
	
	\setcounter{page}{1}
	
	\title[Continuous biframes in Hilbert Spaces]{Continuous biframes in Hilbert Spaces}

	\author[H. Massit, R. Eljazzar, M. Rossafi]{Hafida Massit$^{1}$, Roumaissae Eljazzar$^{1}$ and Mohamed Rossafi$^{2*}$}
	
	\address{$^{1}$Department of Mathematics, University of Ibn Tofail, B.P. 133, Kenitra, Morocco}
	\email{\textcolor[rgb]{0.00,0.00,0.84}{massithafida@yahoo.fr; roumaissae.eljazzar@uit.ac.ma}}
	\address{$^{2}$Department of Mathematics Faculty of Sciences, Dhar El Mahraz University Sidi Mohamed Ben Abdellah, Fes, Morocco}
	\email{\textcolor[rgb]{0.00,0.00,0.84}{rossafimohamed@gmail.com}}

	
	
	\date{
		\newline \indent $^{*}$Corresponding author}

	\begin{abstract} 	In this paper, we present the concept of continuous biframes in a Hilbert space. We examine the essential properties of biframes with an emphasis on the biframe operator. Moreover, we introduce a new type of Riesz bases, referred to as continuous biframe-Riesz bases.
	\end{abstract}
	\maketitle

\baselineskip=12.4pt

	\section{Introduction }

The notion of frames within Hilbert spaces was first proposed by Duffin and Schaffer in 1952 \cite{Duf}, aimed at addressing complex issues in nonharmonic Fourier series. Following the seminal work \cite{DGM} by Daubechies, Grossman, and Meyer, the theory of frames gained extensive application, especially in the specialized areas of wavelet frames and Gabor frames.

The generalization of frames was introduced by G. Kaiser \cite{KAI} and concurrently by Ali, Antoine, and Gazeau \cite{GAZ}, defining them as a family indexed by a locally compact space equipped with a Radon measure, known as continuous frames. Gabardo and Han referred to these as frames associated with measurable spaces in \cite{GHN}. Similarly, Askari-Hemmat, Dehghan, and Radjabalipour named them generalized frames in \cite{AHDR}. In the field of mathematical physics, they are often termed Coherent states \cite{COH}. For additional insights into frames, refer to \cite{Giati, Massit, FR1, RFDCA, r1, r3, r5, r6}.

	By introducing the biframe operator associated with a biframe, we explore biframes' properties through the lens of operator theory. We categorize biframes whose component sequences constitute Bessel sequences, frames, and Riesz bases. Additionally, we derive new bases that bear close relation to orthonormal bases and Riesz bases.
	
	\section{preliminaries}

	In this paper, the symbol $\mathcal{X}$ represents a separable Hilbert space. The expression $ B(\mathcal{X},\mathcal{Y}) $ refers to the collection of all bounded linear operators from $ \mathcal{X} $ to another Hilbert space $ \mathcal{Y} $. When $ \mathcal{X} $ equals $ \mathcal{Y} $, we denote this collection as $ B(\mathcal{X}) $. The identity operator on $ \mathcal{X} $ is represented by $ \mathcal{I} $. Furthermore, $ GL(\mathcal{X}) $ denotes the set of all invertible, bounded linear operators on $ \mathcal{X} $, with $ GL^{+}(\mathcal{X}) $ signifying the subset of $ GL(\mathcal{X}) $ comprising positive operators.

In this work, we present the concept of a mean for a continuous biframe. Our anticipation is that various results from frame theory will extend to these frames. To facilitate this generalization, we first outline some fundamental results and theorems related to these frames. Subsequently, we delve into topics such as the duality of these frames, the perturbation of continuous frames, and the resilience of these frames against the removal of certain elements.
	\begin{definition} \cite{5}
		Consider $\mathcal{X}$ as a complex Hilbert space, and $(\Omega,\mu)$ as a measure space endowed with a positive measure $\mu$.
	
	A mapping $\Xi : \Omega \rightarrow \mathcal{X}$ is defined as a continuous frame in respect to $(\Omega,\mu)$ if it satisfies the following criteria:
	\begin{itemize}
		\item [1]- $\Xi$ is weakly measurable, meaning that for any $\xi \in \mathcal{X}$, the mapping $\omega \rightarrow \langle \xi, \Xi(\omega) \rangle$ is a measurable function over $\Omega$.
		\item [2]- There are two constants $C,D > 0$, ensuring that for every $\xi \in \mathcal{X}$:
		\begin{equation}\label{d1eq1}
		C\|\xi\|^2 \leq \int_{\Omega}|\langle \xi, \Xi(\omega) \rangle|^2 d\mu(\omega) \leq D\|\xi\|^2.
		\end{equation}
	\end{itemize}
	
	\end{definition}


The frame operator $ \mathcal{T}_{\Xi}: \mathcal{X} \rightarrow \mathcal{X} $ is defined by $ \mathcal{T}_{\Xi}(\xi)= \int_{\Omega} \langle \xi, \Xi_{\omega}\rangle \Xi_{\omega} d\mu(\omega) $
and is a self-adjoint operator that belongs to $ GL^{+}(\mathcal{X}) $.

It is noteworthy that two Bessel sequences $\{ \Xi_{\omega} \}$ and $ \{\Phi_{\omega}\} $ are dual frames for $\mathcal{X}$ if any of the following conditions is satisfied:
\begin{itemize}
	\item [(i)] $ \xi= \int_{\Omega} \langle \xi,\Phi_{\omega}\rangle \Xi_{\omega}  d\mu(\omega),\;\forall \xi \in\mathcal{X}$.
	\item[(ii)]$ \xi= \int_{\Omega} \langle \xi,\Xi_{\omega}\rangle \Phi_{\omega}  d\mu(\omega),\;\forall \xi \in\mathcal{X}$.
	\item[(iii)] $ \langle \xi,\eta \rangle = \int_{\Omega} \langle \xi,\Xi_{\omega} \rangle \langle \Phi_{\omega}, \eta \rangle d\mu(\omega),\;\forall \xi,\eta \in\mathcal{X} $.
\end{itemize}

A Riesz basis in the Hilbert space $\mathcal{X}$ is characterized as a set $\{Ue_{\omega}\}$, where $\{e_{\omega}\}$ denotes an orthonormal basis for $\mathcal{X}$, and $U$ is a member of $GL(\mathcal{X})$.

Additionally, the sequences $\{\Xi_{\omega}\}$ and $\{\Phi_{\omega}\}$ in $\mathcal{X}$ are termed biorthogonal if $\langle \Xi_{i}, \Phi_{j} \rangle = \delta_{i,j}$, signifying orthogonality and normalization between corresponding elements of the two sequences.


\begin{definition}\cite{PDA}
	In a measure space $(\Omega,\mu)$ endowed with a positive measure $\mu$ and given $\mathcal{P}\in GL(\mathcal{X})$, a map $\Xi:\Omega\rightarrow \mathcal{X}$ is termed a $\mathcal{P}$-controlled continuous frame if it meets the following condition: there exist constants $0<C\leq D<\infty$ such that for every $\xi\in \mathcal{X}$, the inequality
	\begin{equation*}
	C\|\xi\|^{2}\leq \int_{\Omega}\langle \xi,\Xi(w)\rangle \langle \mathcal{P}\Xi(w),\xi\rangle \ d\mu(w)\leq D\|\xi\|^{2}
	\end{equation*}
	holds true.
\end{definition}

\begin{definition}\cite{PDA}
	Consider $(\Omega,\mu)$ as a measure space equipped with a positive measure $\mu$ and $\mathcal{P}, \mathcal{Q} \in GL(\mathcal{X})$. A $(\mathcal{P}, \mathcal{Q})$-controlled continuous frame is a function $\Xi:\Omega\rightarrow \mathcal{X}$ satisfying the condition that for constants $0<C\leq D<\infty$, the following inequality is true for all $\xi \in \mathcal{X}$:
	\begin{equation*}
	C\|\xi\|^{2}\leq \int_{\Omega}\langle \xi,\mathcal{P}\Xi(w)\rangle \langle \mathcal{Q}\Xi(w),\xi\rangle \ d\mu(w)\leq D\|\xi\|^{2}.
	\end{equation*}
\end{definition}

\begin{theorem}\cite{CH}\label{ta} 
	Suppose $\mathcal{X}$ is a Banach space and $\mathcal{T} : \mathcal{X} \rightarrow \mathcal{X}$ is a bounded operator with $\Vert I - \mathcal{T} \Vert < 1$. Under these conditions, $\mathcal{T}$ is guaranteed to be invertible.
\end{theorem}

\begin{proposition}\cite{Musazadeh} \label{p1} 
	Let $ \mathcal{S}_{1}, \mathcal{S}_{2} \in GL^{+}(\mathcal{X})$. Then $ \mathcal{S}_{2} = U \mathcal{S}_{1} \mathcal{V}^{\ast} $ if and only if $ \mathcal{V} = \mathcal{S}_{2}^{c} \mathcal{W} \mathcal{S}_{1}^{-a}$ and $ U = \mathcal{S}_{2}^{d} \mathcal{T} \mathcal{S}_{1}^{-b}$ where $ \mathcal{W} $ and $ \mathcal{T} $ are bounded operators on $\mathcal{X}$ with $ \mathcal{T}\mathcal{W}^{\ast} = I $, and $ a, b, c, d \in \mathbb{R} $ satisfying $ c + d = 1, a + b = 1 $.
\end{proposition}

	\section{Main result}

	\begin{definition}\label{d1}
		Consider $\mathcal{X}$ as a complex Hilbert space and $(\Omega,\mu)$ as a measure space endowed with a positive measure $\mu$. The functions $\Xi, \Phi : \Omega \rightarrow \mathcal{X}$ are designated as continuous biframes with respect to $(\Omega,\mu)$ if they satisfy the following criteria:
		\begin{itemize}
			\item [(1)] Both $\Xi$ and $\Phi$ are weakly-measurable, meaning for every $\xi \in \mathcal{X}$, the mappings $\omega \rightarrow \langle \xi, \Xi_{\omega}\rangle$ and $\omega \rightarrow \langle \xi, \Phi_{\omega}\rangle$ are measurable over $\Omega$.
			\item[(2)] There exist positive constants $C$ and $D$ such that for all $\xi \in \mathcal{X}$:
			\begin{equation}\label{1}
			C \|\xi\|^2 \leq \int_{\Omega}\langle \xi,\Xi_{\omega}\rangle \langle \Phi_{\omega},\xi\rangle d\mu(\omega)\leq D\|\xi\|^2
			\end{equation}
		\end{itemize}
		
		The constants $C$ and $D$ are termed as the lower and upper biframe bounds, respectively. The biframe is considered Parseval if $C = D = 1$.
	\end{definition}

\begin{remark}
	Based on Definition \ref{d1}, we observe:
	\begin{itemize}
		\item[(i)] If $(\Xi, \Xi)$ forms a continuous biframe for $\mathcal{X}$, then $\Xi$ constitutes a continuous frame for $\mathcal{X}$.
		\item[(ii)] If $(\Xi,U\Xi)$ is a continuous biframe for some $U\in GL(\mathcal{X})$, then $\Xi$ serves as a $U$-controlled integral frame for $\mathcal{X}$.
		\item[(iii)] If $(\mathcal{T}\Xi, U\Xi)$ constitutes a continuous biframe for certain $\mathcal{T}$ and $U$ in $GL(\mathcal{X})$, then $\Xi$ is a $(\mathcal{T}, U)$-controlled integral frame for $\mathcal{X}$.
	\end{itemize}
\end{remark}

\begin{example}
	Take into account the Hilbert space $\mathbb{R}^{2}$ and the interval $\Omega = [0,1]$, which is furnished with Lebesgue's measure, thereby forming a measure space. Consider the following mappings:
	\begin{align*}
	\Xi:& [0,1] \rightarrow \mathbb{R}^{2} \\
	&\omega \mapsto (0,\omega)
	\end{align*}
	and 
	\begin{align*}
	\Phi:& [0,1] \rightarrow \mathbb{R}^{2} \\
	&\omega \mapsto (\omega,1)
	\end{align*}
	From these definitions, it follows that for any point $(x, y)$ in $\mathbb{R}^{2}$:
	\begin{align*}
	\frac{x^{2} + y^{2}}{3} \leq \int_{\Omega} \langle (x,y), (0,\omega)\rangle \langle (\omega,1), (x,y)\rangle d\mu(\omega) \leq \frac{x^{2} + y^{2}}{2}
	\end{align*}
	Consequently, the pair $(\Xi, \Phi)$ constitutes a continuous biframe for $\mathbb{R}^{2}$ with lower and upper bounds of $\frac{1}{3}$ and $\frac{1}{2}$, respectively.
\end{example}

	We establish the concept of a biframe operator associated with a biframe and define biframes through the examination of this operator's characteristics.

\begin{definition}
	Given a continuous biframe $(\Xi, \Phi)$ with respect to $(\Omega,\mu)$, we define the biframe operator $\mathcal{T}_{\Xi, \Phi}$ as follows:
	\begin{equation*} 
	\mathcal{T}_{\Xi, \Phi}:\mathcal{X} \rightarrow \mathcal{X}
	\end{equation*}
	\begin{equation}  
	\mathcal{T}_{\Xi, \Phi}(\xi)=\int_{\Omega}\langle \xi,\Xi_{\omega}\rangle  \Phi_{\omega} d\mu (\omega)\;\;\forall \xi\in \mathcal{X}
	\end{equation}
\end{definition}

	We introduce several characteristics of the continuous biframe operator.

\begin{theorem}\label{t4}
	Consider a continuous biframe $(\Xi, \Phi)$ within the Hilbert space $\mathcal{X}$, characterized by the bounds $C$ and $D$. Under these conditions, the subsequent assertions hold:
	\begin{itemize}
		\item [(i)] The biframe operator $\mathcal{T}_{\Xi, \Phi}$ exhibits the qualities of being bounded, positive, and invertible.
		\item [(ii)] The pair $(\Xi, \Phi)$ constitutes a continuous biframe precisely when $(\Phi, \Xi)$ similarly forms a continuous biframe.
	\end{itemize}
\end{theorem}

\begin{proof}
	(i) Consider the expression:
	\begin{align*}
	\langle \mathcal{T}_{\Xi, \Phi}\xi, \xi \rangle & = \langle \int_{\Omega} \langle \xi, \Xi_{\omega} \rangle \Phi_{\omega} d\mu(\omega), \xi \rangle \\
	& = \int_{\Omega} \langle \xi, \Xi_{\omega} \rangle \langle \Phi_{\omega}, \xi \rangle d\mu(\omega)
	\end{align*}
	Based on \ref{1}, we have 
	\begin{equation*}
	0 \leq C\|\xi\|^2 \leq \langle \mathcal{T}_{\Xi, \Phi}\xi, \xi \rangle \;\; \forall \xi \in \mathcal{X}.
	\end{equation*}
	Hence, $\mathcal{T}_{\Xi, \Phi}$ is a positive and self-adjoint operator.
	
	Again, from \ref{1}, we derive 
	\begin{equation*}
	C\|\xi\|^2 \leq \langle \mathcal{T}_{\Xi, \Phi}\xi, \xi \rangle \leq D\|\xi\|^2 \;\; \forall \xi \in \mathcal{X}.
	\end{equation*}
	Thus,
	\begin{equation*}
	C.I \leq \mathcal{T}_{\Xi, \Phi} \leq D.I
	\end{equation*}
	indicating that $\mathcal{T}_{\Xi, \Phi}$ is a bounded operator.
	
	Moreover,
	\begin{equation*}
	0 \leq I - D^{-1}\mathcal{T}_{\Xi, \Phi} \leq \frac{D-C}{D}.I
	\end{equation*}
	As a result,
	\begin{equation*}
	\| I - D^{-1}\mathcal{T}_{\Xi, \Phi} \| = \sup_{\xi \in \mathcal{X}, \|\xi\| = 1} \| \langle I - D^{-1}\mathcal{T}_{\Xi, \Phi}\xi, \xi \rangle \| \leq \frac{D-C}{D} < 1
	\end{equation*}
	Applying Theorem \ref{ta}, we conclude that $\mathcal{T}_{\Xi, \Phi}$ is invertible.
	
	(ii) Assume $(\Xi, \Phi)$ is a continuous biframe with bounds $C$ and $D$. For every $\xi \in \mathcal{X}$,
	\begin{equation*}
	C \|\xi\|^2 \leq \int_{\Omega}\langle \xi, \Xi_{\omega}\rangle \langle \Phi_{\omega}, \xi\rangle d\mu(\omega) \leq D\|\xi\|^2,
	\end{equation*}
	We can express this as
	\begin{equation*}
	\int_{\Omega}\langle \xi, \Xi_{\omega}\rangle \langle \Phi_{\omega}, \xi\rangle d\mu(\omega) = \int_{\Omega}\langle \xi, \Phi_{\omega}\rangle \langle \Xi_{\omega}, \xi\rangle d\mu(\omega),
	\end{equation*}
	which implies that
	\begin{equation*}
	C \|\xi\|^2 \leq \int_{\Omega}\langle \xi, \Phi_{\omega}\rangle \langle \Xi_{\omega}, \xi\rangle d\mu(\omega) \leq D\|\xi\|^2,
	\end{equation*}
	Therefore, $(\Phi, \Xi)$ is also a continuous biframe with bounds $C$ and $D$. The converse can be demonstrated similarly.
\end{proof}

\begin{example}
	Let $\mathcal{X} = \{ M = \begin{pmatrix}
	a & 0  \\ 
	0 & b
	\end{pmatrix}; a,b \in \mathbb{C} \}$, $\Omega =[0,1]$, and 
	
	define $\Xi:[0,1]\rightarrow \mathcal{X}$ by $\Xi(\omega) = \begin{pmatrix}
	\omega & 0  \\ 
	0 & \frac{\omega}{2}
	\end{pmatrix}$
	
	and
	
	define $\Phi:[0,1]\rightarrow \mathcal{X}$ by $\Phi(\omega) = \begin{pmatrix}
	2\omega & 0  \\ 
	0 & \omega
	\end{pmatrix}$
	
	For $M \in \mathcal{X}$, we have
	
	$\int_{\Omega} \langle M, \Xi(\omega)\rangle \langle \Phi(\omega), M \rangle d\mu(\omega) = \frac{1}{3} \begin{pmatrix}
	2a^{2} & 0  \\ 
	0 & \frac{b^{2}}{2}
	\end{pmatrix}$
	
	This demonstrates that $(\Xi, \Phi)$ forms a continuous biframe for $\mathcal{X}$ with bounds $\frac{1}{2}$ and $2$.
\end{example}

\begin{corollary}
	Every continuous biframe $(\Xi, \Phi)$ forms a pair frame.
\end{corollary}

\begin{proof}
	By Theorem \ref{t4}, the operator $\mathcal{T}_{\Xi, \Phi}$ is well-defined and invertible. Hence, according to the definition of pair frames, $(\Xi, \Phi)$ constitutes a pair frame.
\end{proof}

Among the fundamental results in frame theory is the reconstruction of elements in a Hilbert space from frame coefficients. The following theorem provides this reconstruction using biframes.

\begin{theorem}
	Let $(\Xi, \Phi)$ be a continuous biframe for $\mathcal{X}$ with the biframe operator $\mathcal{T}_{\Xi, \Phi}$. Then, for every $\xi \in \mathcal{X}$, the following reconstruction formula is valid:
	\begin{equation}\label{3}
	\xi = \int_{\Omega} \langle \xi, \mathcal{T}^{-1}_{\Phi, \Xi} \Xi_{\omega} \rangle \Phi_{\omega} d\mu(\omega) = \int_{\Omega} \langle \xi, \Xi_{\omega} \rangle \mathcal{T}^{-1}_{\Xi, \Phi} \Phi_{\omega} d\mu(\omega)
	\end{equation}
\end{theorem}

\begin{proof}
	For any element $\xi \in \mathcal{X}$, it holds that:
	\begin{equation*}
	\xi = \mathcal{T}^{-1}_{\Xi, \Phi} \mathcal{T}_{\Xi, \Phi}\xi = \mathcal{T}^{-1}_{\Xi, \Phi}\int_{\Omega}\langle \xi, \Xi_{\omega} \rangle  \Phi_{\omega} d\mu (\omega) = \int_{\Omega}\langle \xi, \Xi_{\omega} \rangle \mathcal{T}^{-1}_{\Xi, \Phi}\Phi_{\omega} d\mu (\omega)
	\end{equation*}
	and
	\begin{equation*}
	\xi = \mathcal{T}_{\Xi, \Phi} \mathcal{T}^{-1}_{\Xi, \Phi}\xi = \int_{\Omega}\langle \mathcal{T}^{-1}_{\Xi, \Phi} \xi, \Xi_{\omega} \rangle  \Phi_{\omega} d\mu (\omega) = \int_{\Omega}\langle \xi, \mathcal{T}^{-1}_{\Phi, \Xi}\Xi_{\omega} \rangle \Phi_{\omega} d\mu (\omega)
	\end{equation*}
	In a complex Hilbert space, every positive operator is self-adjoint. Therefore, the positive operator $\mathcal{T}_{\Xi, \Phi}$ possesses the self-adjointness property in complex Hilbert spaces.
\end{proof}

\begin{example}
	In the context of the space $\mathcal{X} = \mathbb{R}^{2}$, consider the vector $M = \begin{pmatrix} x \\ y \end{pmatrix}$ within $\mathbb{R}^{2}$. Evaluating the inner product gives us:
	
	$\langle \begin{pmatrix} 2 & 1 \\ 3 & 4 \end{pmatrix} M, M \rangle = \langle \begin{pmatrix} 2x + y \\ 3x + 4y \end{pmatrix}, M \rangle = 5x^{2} + 4y^{2} + xy > 0$.
	
	This calculation demonstrates that the matrix $N = \begin{pmatrix} 2 & 1 \\ 3 & 4 \end{pmatrix}$ is positive. However, it is evident that $N$ does not possess the property of self-adjointness.
\end{example}

The reconstruction formula in continuous biframe theory, as delineated in \ref{3}, designates the sequence $\{\langle \xi, \mathcal{T}^{-1}_{\Xi, \Phi}\Xi_{\omega} \rangle\}$ as the continuous biframe coefficients sequence.

Our objective is to characterize continuous biframes by examining the attributes of the continuous biframe operator.

\begin{theorem} \label{t8}
	Suppose $\{\Xi_{\omega}\}_{\omega\in\Omega}$ and $\{\Phi_{\omega}\}_{\omega\in\Omega}$ are sequences within a complex Hilbert space $\mathcal{X}$. Then, the set $\{\{\Xi_{\omega}\}_{\omega\in\Omega}, \{\Phi_{\omega}\}_{\omega\in\Omega}\}$ forms a continuous biframe for $\mathcal{X}$ if and only if the operator $\mathcal{T}_{\Xi, \Phi}$ is both positive and bounded.
\end{theorem}
	
	\begin{proof}
		Assume that the pair $\left(\{\Xi_{\omega}\}_{\omega\in\Omega}, \{\Phi_{\omega}\}_{\omega\in\Omega}\right)$ constitutes a continuous biframe for $\mathcal{X}$. By Theorem \ref{t4}, $\mathcal{T}_{\Xi, \Phi}$ is a positive and invertible operator, and thus, it is also bounded.
		
		Now, suppose $\mathcal{T}_{\Xi, \Phi}$ is a positive and bounded operator on $\mathcal{X}$. Given the positivity of $\mathcal{T}_{\Xi, \Phi}$, it follows that this operator is self-adjoint. For any $\xi \in \mathcal{X}$:
		\begin{equation*}
		\int_{\Omega} \langle \xi, \Xi_{\omega} \rangle \langle \Phi_{\omega}, \xi \rangle d\mu(\omega) = \langle \mathcal{T}_{\Xi, \Phi}\xi, \xi \rangle = | \langle \mathcal{T}_{\Xi, \Phi}\xi, \xi \rangle | \leq \| \mathcal{T}_{\Xi, \Phi} \| \| \xi \|^2
		\end{equation*}
		From the positivity of $\mathcal{T}_{\Xi, \Phi}$, we derive
		\begin{equation}\label{4}
		\| \mathcal{T}_{\Xi, \Phi} \xi \|^2 \leq \| \mathcal{T}_{\Xi, \Phi} \| \langle \mathcal{T}_{\Xi, \Phi}, \xi \rangle
		\end{equation}
		$\mathcal{T}_{\Xi, \Phi}$ is bounded with a lower bound $a$:
		\begin{equation}\label{5}
		\exists a > 0, \; a\| \xi \| \leq \| \mathcal{T}_{\Xi, \Phi} \xi \|
		\end{equation}
		Utilizing \ref{4} and \ref{5}, we conclude
		\begin{equation*}
		a^2 \| \xi \|^2 \leq \| \mathcal{T}_{\Xi, \Phi} \| \langle \mathcal{T}_{\Xi, \Phi}, \xi \rangle, \; \forall \xi \in \mathcal{X} \Rightarrow \langle \mathcal{T}_{\Xi, \Phi}, \xi \rangle \geq \frac{a^2}{\| \mathcal{T}_{\Xi, \Phi} \|} \| \xi \|^2
		\end{equation*}
		Thus,
		\begin{equation*}
		\frac{a^2}{\| \mathcal{T}_{\Xi, \Phi} \|} \| \xi \|^2 \leq \int_{\Omega} \langle \xi, \Xi_{\omega} \rangle \langle \Phi_{\omega}, \xi \rangle d\mu(\omega) \leq \| \mathcal{T}_{\Xi, \Phi} \| \| \xi \|^2
		\end{equation*}
		
		Therefore, the pair $\left(\{\Xi_{\omega}\}_{\omega\in\Omega}, \{\Phi_{\omega}\}_{\omega\in\Omega}\right)$ forms a continuous biframe for $\mathcal{X}$.
	\end{proof}

\begin{theorem}\label{t3.10}
	Assume that the pair $(\{\Xi_{\omega}\}_{\omega\in\Omega}, \{\Phi_{\omega}\}_{\omega\in\Omega})$ forms a continuous biframe for the Hilbert space $\mathcal{X}$.  $\{\Xi_{\omega}\}_{\omega \in \Omega }$ constitutes a Riesz basis for $\mathcal{X}$ if and only if  $\{\Phi_{\omega}\}_{\omega \in \Omega}$ also constitutes a Riesz basis for $\mathcal{X}$.
\end{theorem}

\begin{proof}
	Consider the case where $\{\Xi_{\omega}\}_{\omega\in\Omega}$ is a Riesz basis. There exists a bijective linear operator $U$ belonging to $GL(\mathcal{X})$, satisfying $\Xi_{\omega} = U e_{\omega}$ for every $\omega$ in $\Omega$, with $\{e_{k}\}_{k \in \Omega}$ being an orthonormal basis of $\mathcal{X}$.
	
	Now, define the linear operator
	\begin{equation*}
	S: \mathcal{X} \rightarrow \mathcal{X}, \quad S\xi = \mathcal{T}_{\Xi,\Phi}(U^{\ast})^{-1} \xi
	\end{equation*}
	For each $k$ in $\Omega$, we have
	
	\begin{align*}
	S e_{k} = \mathcal{T}_{\Xi,\Phi}(U^{\ast})^{-1}e_{k} &= \int_{\Omega} \langle (U^{\ast})^{-1}e_{k}, \Xi_{\omega} \rangle \Phi_{\omega} d\mu (\omega)\\
	&= \int_{\Omega} \langle (U^{\ast})^{-1}e_{k}, U e_{\omega} \rangle \Phi_{\omega} d\mu(\omega)\\
	&= \int_{\Omega} \langle U^{\ast} (U^{\ast})^{-1}e_{k}, e_{\omega} \rangle \Phi_{\omega} d\mu(\omega)\\
	&= \int_{\Omega} \langle e_{k}, e_{\omega} \rangle \Phi_{\omega} d\mu (\omega)\\
	&= \Phi_{k}
	\end{align*}
	
	Thus, it is established that $\{\Phi_{\omega} \}_{\omega \in \Omega}$ forms a Riesz basis for $\mathcal{X}$. Similarly, one can demonstrate that $\{\Xi_{\omega}\}_{\omega\in\Omega}$ is a Riesz basis if $\{\Phi_{\omega}\}_{\omega\in\Omega}$ is a Riesz basis.
\end{proof}

\begin{remark}
	A positive operator within a complex Hilbert space $ \mathcal{X} $ is self-adjoint. This characteristic allows for the analysis of the effects exerted by two distinct operators on sequences within a biframe. This exploration utilizes the self-adjointness property of the positive, continuous biframe operator $ \mathcal{T}_{\Xi,\Phi} $.
\end{remark}

\begin{theorem}\label{t3.12}
	Assuming that $(\Xi, \Phi) = (\{\Xi_{\omega}\}_{\omega \in \Omega}, \{\Phi_{\omega}\}_{\omega \in \Omega})$ forms a biframe for $\mathcal{X}$ with the continuous biframe operator $\mathcal{T}_{\Xi,\Phi}$, the following statements hold true:
	\begin{itemize}
		\item [(i)] The tuple $(\{S\Xi_{\omega}\}_{\omega \in \Omega}, \{U\Phi_{\omega}\}_{\omega \in \Omega})$ constitutes a biframe for $\mathcal{X}$, with certain operators $S$ and $U$ in $B(\mathcal{X})$, if and only if there exist operators $\mathcal{Q} \in B^{+}_{b,b} (\mathcal{X})$ and $V, W \in B(\mathcal{X})$, satisfying $\mathcal{T}W^{\ast} = I$ and $S = \mathcal{Q}^{c}W\mathcal{T}^{-a}_{\Xi,\Phi}$, $U = \mathcal{Q}^{d} V\mathcal{T}^{-b}_{\Xi,\Phi}$, for $a, b, c, d \in \mathbb{R}$ with $a+b=1$ and $c+d=1$.
		\item [(ii)] The pair $(\{S\Xi_{\omega}\}_{\omega \in \Omega}, \{U\Phi_{\omega}\}_{\omega \in \Omega})$ serves as a continuous Parseval biframe for $\mathcal{X}$, with some operators $S$ and $U$ in $B(\mathcal{X})$, if and only if operators $V, W \in B(\mathcal{X})$ exist, ensuring $V W^{\ast} = I$ and $S = W\mathcal{T}^{-a}_{\Xi,\Phi}$, $U = V\mathcal{T}^{-b}_{\Xi,\Phi}$, where $a, b \in \mathbb{R}$ with $a+b=1$.
	\end{itemize}
	Particularly, $(\{\Xi_{\omega}\}_{\omega \in \Omega}, \{\Phi_{\omega}\}_{\omega \in \Omega})$ is a Parseval biframe if and only if $U S^{\ast} = I$.
\end{theorem}

\begin{proof}
	(i) Assume that $ (S\Xi, U\Phi) = (\{S\Xi_{\omega}\}_{\omega \in \Omega}, \{U\Phi_{\omega}\}_{\omega \in \Omega}) $ constitutes a continuous biframe for $ \mathcal{X} $ with the continuous biframe operator $ \mathcal{T}_{S\Xi, U\Phi} $, where $ S, U \in B(\mathcal{X}) $. For an element $ \xi \in \mathcal{X} $,
	\begin{equation}\label{6}
	\mathcal{T}_{S\Xi, U\Phi}\xi = \int_{\Omega} \langle \xi, S\Xi_{\omega} \rangle U\Phi_{\omega} d\mu (\omega) = U \int_{\Omega}\langle S^{\ast}\xi, \Xi_{\omega} \rangle \Phi_{\omega} d\mu (\omega) = U \mathcal{T}_{\Xi, \Phi}S^{\ast}\xi.
	\end{equation}
	This leads to the conclusion that the operators $ \mathcal{T}_{\Xi, \Phi} $ and $ \mathcal{T}_{S\Xi, U\Phi} $ belong to $ B^{+}_{b,b}(\mathcal{X}) $ and also to $ GL^{+}(\mathcal{X}) $.
	
	By reference \ref{6} and Proposition \ref{p1}, there exist operators $ V, W \in B(\mathcal{X}) $ such that $ VW^{\ast} = I $ and
	\begin{equation}
	S = \mathcal{T}^{r}_{S\Xi, U\Phi} W \mathcal{T}^{-a}_{\Xi, \Phi}, \quad U = \mathcal{T}^{t}_{S\Xi, U\Phi} V \mathcal{T}^{-b}_{\Xi, \Phi}.
	\end{equation}
	Conversely, let $ \mathcal{Q} \in B^{+}_{b,b}(\mathcal{X}) $, and suppose $ T $ and $ W $ are operators in $ B(\mathcal{X}) $ such that $ TW^{\ast} = I $ and
	\begin{equation}
	S = \mathcal{Q}^{r} W \mathcal{T}^{-a}_{\Xi, \Phi}, \quad U = \mathcal{Q}^{t} V \mathcal{T}^{-b}_{\Xi, \Phi}.
	\end{equation}
	
	Consequently, $ \mathcal{Q} = U \mathcal{T}_{\Xi, \Phi} S^{\ast} $ and
	\begin{equation*}
	\mathcal{Q}\xi = U \mathcal{T}_{\Xi, \Phi}S^{\ast}\xi = U \int_{\Omega} \langle S^{\ast}\xi, \Xi_{\omega} \rangle \Phi_{\omega} d\mu(\omega) =
	\int_{\Omega}\langle \xi, S\Xi_{\omega} \rangle U\Phi_{\omega} d\mu(\omega).
	\end{equation*}
	This implies that $ \mathcal{Q} $ is part of $ B^{+}_{b,b}(\mathcal{X}) $.
	
	For part (ii), it directly follows from part (i).
\end{proof}

\begin{corollary}\label{c13}
	Assume $ S, U \in B(\mathcal{X}) $. Given the sequences $\{\Xi_{\omega}\}_{\omega \in \Omega}$ and $\{\Phi_{\omega}\}_{\omega \in \Omega}$ within $\mathcal{X}$, the following statements are equivalent: $(\{S\Xi_{\omega}\}_{\omega \in \Omega}, \{U\Phi_{\omega}\}_{\omega \in \Omega})$ forms a Parseval continuous biframe if and only if $ US^{\ast} = I$.
	
	\begin{itemize}
		\item[(i)] The sequences $\{\Xi_{\omega}\}_{\omega \in \Omega}$ and $\{\Phi_{\omega}\}_{\omega \in \Omega}$ constitute dual frames in $\mathcal{X}$.
		\item[(ii)] The sequences $\{\Xi_{\omega}\}_{\omega \in \Omega}$ and $\{\Phi_{\omega}\}_{\omega \in \Omega}$ are biorthogonal frames for $\mathcal{X}$.
	\end{itemize}     
\end{corollary}

\begin{proof}
	Assuming that the sequences $\{\Xi_{\omega}\}_{\omega \in \Omega}$ and $\{\Phi_{\omega}\}_{\omega \in \Omega}$ are biorthogonal frames in $\mathcal{X}$, consider an element $\xi \in \mathcal{X}$. The following equalities then hold:
	\begin{align*}
	\int_{\Omega} \langle \xi, \Xi_{\omega} \rangle \langle \Phi_{\omega}, \xi \rangle d\mu(\omega) &= \int_{\Omega} \langle \xi, \Xi_{\omega} \rangle \langle \Phi_{\omega}, \sum_{k=1}^{\infty} \langle \xi, \mathcal{T}^{-1}_{\Xi,\Phi}\Xi_{k} \rangle \Xi_{k} \rangle d\mu(\omega)\\
	&= \int_{\Omega} \sum_{k=1}^{\infty}\langle \xi, \Xi_{\omega} \rangle \langle \mathcal{T}^{-1}_{\Xi,\Phi}\Xi_{k}, \xi \rangle \langle \Phi_{\omega}, \Xi_{k} \rangle d\mu (\omega) \\
	&= \int_{\Omega} \langle \xi, \Xi_{\omega} \rangle \langle \mathcal{T}^{-1}_{\Xi,\Phi}\Xi_{\omega}, \xi \rangle d\mu (\omega) \\
	&= \Vert \xi \Vert^{2}.
	\end{align*}
\end{proof}

\begin{remark}
	The set $(\{Se_{\omega}\}_{\omega\in\Omega}, \{Ue_{\omega}\}_{\omega\in\Omega})$ forms a Parseval continuous biframe if and only if the product $US^{\ast} = I$ for $\{e_{\omega}\}_{\omega\in \Omega}$, which is an orthonormal basis for $\mathcal{X}$.
\end{remark}

We now focus on characterizing the biframes where the sequences are Bessel sequences, frames, and Riesz bases.

\begin{theorem}
	Bessel sequences $\Xi = \{\Xi_{\omega}\}_{\omega\in\Omega}$ and $\Phi = \{\Phi_{\omega}\}_{\omega\in\Omega}$ together form a continuous biframe for $\mathcal{X}$ if and only if there are operators $\mathcal{Q} \in B^{+}(b,b) (\mathcal{X})$ and $S, W \in B(\mathcal{X})$ such that $SW^{\ast} = I$, $\Xi_{\omega} = \mathcal{Q}^{c} W e_{\omega}$, and $\Phi_{\omega} = \mathcal{Q}^{d} S e_{\omega}$ for every $\omega\in\Omega$, where $c, d \in \mathbb{R}$ with $c + d = 1$, and $E = \{e_{\omega}\}_{\omega\in \Omega}$ is an orthonormal basis for $\mathcal{X}$.
\end{theorem}

\begin{proof}
	Consider the pair $(\Xi, \Phi) = \{\Xi_{ \omega}\}_{\omega \in \Omega}, \{\Phi_{\omega}\}_{\omega \in \Omega}$ as a continuous biframe with a continuous biframe operator $\mathcal{T}_{\Xi,\Phi}$, composed of Bessel sequences $\{\Xi_{\omega}\}_{\omega \in \Omega}$ and $\{\Phi_{\omega}\}_{\omega \in \Omega}$. Since both $\{\Xi_{\omega}\}_{\omega \in \Omega}$ and $\{\Phi_{\omega}\}_{\omega \in \Omega}$ are Bessel sequences, there exist operators $S, W \in B(\mathcal{X})$ and a set $E = \{e_{\omega}\}_{\omega\in\Omega}$ forming an orthonormal basis for $\mathcal{X}$, such that 
	\begin{equation}
	\Xi_{\omega} = Se_{\omega} = SE \quad \text{and} \quad \Phi_{\omega} = We_{\omega} = WE \quad \text{for all} \quad \omega \in \Omega.
	\end{equation}    
	Referring to Theorem \ref{t3.12}, we find $V, W \in B(\mathcal{X})$ such that $VW^{\ast} = I$. Furthermore, by \ref{6},
	\begin{equation}
	S = \mathcal{T}^{c}_{SE, WE} W, \quad W = \mathcal{T}^{d}_{SE, WE} V \quad \text{for all} \quad c, d \in \mathbb{R}, \quad c + d = 1.
	\end{equation}
	This leads to
	\begin{equation*}
	\Xi_{\omega} = Se_{\omega} = \mathcal{T}^{c}_{SE, WE} W e_{\omega} = \mathcal{T}^{c}_{\Xi, \Phi} W e_{\omega}
	\end{equation*}
	and
	\begin{equation}
	\Phi_{\omega} = We_{\omega} = \mathcal{T}^{d}_{SE, WE} V e_{\omega} = \mathcal{T}^{d}_{\Xi, \Phi} V e_{\omega}, \quad \text{for all} \quad \omega \in \Omega.    
	\end{equation}
	Assume the existence of an orthonormal basis $\{e_{\omega}\}_{\omega}$ and operators $\mathcal{Q} \in B^{+}(C, D)(\mathcal{X})$ and $V, W \in \mathcal{X}$ such that $VW^{\ast} = I$, $\Xi_{\omega} = \mathcal{Q}^{c} V e_{\omega}$, and $\Phi_{\omega} = \mathcal{Q}^{d} V e_{\omega}$ for every $\omega \in \Omega$, for all $c, d \in \mathbb{R}, \; c + d = 1$.
	
	For any $\xi \in \mathcal{X}$, it follows that
	\begin{align*}
	\int_{\Omega} \langle \xi, \Xi_{\omega} \rangle \Phi_{\omega} \, d\mu(\omega) &= \int_{\Omega} \langle \xi, \mathcal{Q}^{c} W e_{\omega} \rangle \mathcal{Q}^{d} V e_{\omega} \, d\mu(\omega) \\
	&= \mathcal{Q}^{d} V \int_{\Omega} \langle W^{\ast} \mathcal{Q}^{c} \xi, e_{\omega} \rangle e_{\omega} \, d\mu(\omega) \\
	&= \mathcal{Q}^{d} V W^{\ast} \mathcal{Q}^{c} \xi \\
	&= \mathcal{Q} \xi.
	\end{align*}
\end{proof}

\begin{corollary}
	The Bessel sequences $\{\Xi_{\omega}\}$ and $\{\Phi_{\omega}\}$ form a continuous Parseval biframe for $\mathcal{X}$ if and only if there exist operators $S, W \in B(\mathcal{X})$ such that $SW^{\ast} = I$, $\Xi_{\omega} = W e_{\omega}$, and $\Phi_{\omega} = S e_{\omega}$ for every $\omega \in \Omega$, whenever $\{e_{\omega}\}_{\omega}$ is an orthonormal basis for $\mathcal{X}$.
	
	This implies that $\{\Xi_{\omega}\}$ and $\{\Phi_{\omega}\}$ are dual frames.
\end{corollary}

We revisit the notion of a generalized dual for continuous frames.

\begin{definition}
	Let $\{\Phi_{\omega}\}$ be a $g-$dual frame of the frame $\{\Xi_{\omega}\}_{\omega}$ for $\mathcal{X}$ if there exists an invertible operator $A \in B(\mathcal{X})$ such that, for all $\xi \in \mathcal{X}$, 
	the equality $\xi = \int_{\Omega} \langle A\xi, \Phi_{\omega} \rangle \Xi_{\omega} \, d\mu(\omega)$ is valid.
	If $A = I$, then $\{\Phi_{\omega}\}$ is an ordinary dual frame of $\{\Xi_{\omega}\}$, and we express this as $\xi = \int_{\Omega} \langle \xi, \Phi_{\omega} \rangle \Xi_{\omega} \, d\mu(\omega)$.
\end{definition}

\begin{theorem}
	Let $\{\Xi_{\omega}\}$ be a frame for $\mathcal{X}$ with the frame operator $\mathcal{T}_{\Xi}$. The following statements are equivalent:
	\begin{itemize}
		\item [(i)] The sequence $\{\Phi_{\omega} \}$ is a Bessel sequence for $\mathcal{X}$, and the pair $(\{\Xi_{\omega}\}_{\omega}, \{\Phi_{\omega}\}_{\omega})$ forms a continuous biframe for $\mathcal{X}$.
		\item[(ii)] $\{\Phi_{\omega}\} = \{ (\mathcal{T}_{\Xi}\mathcal{Q})^{-1}\Xi_{\omega} + \eta_{\omega} - \sum_{k=1}^{\infty} \langle \mathcal{T}^{-1}_{\Xi}\Xi_{\omega}, \Xi_{k} \rangle \eta_{k}\}$, where $\{\eta_{\omega}\}_{\omega}$ is a Bessel sequence in $\mathcal{X}$, and $\mathcal{Q}$ is an element of $B^{+}_{a,a}(\mathcal{X})$.
		\item[(iii)] The sequence $\{\Phi_{\omega}\}_{\omega}$ is a frame for $\mathcal{X}$, and the pair $(\{\Xi_{\omega}\}_{\omega}, \{\Phi_{\omega}\}_{\omega})$ forms a continuous biframe for $\mathcal{X}$.
	\end{itemize}
\end{theorem}

\begin{proof}
	For $(i)\Rightarrow (ii)$, assume that $\{\Phi_{\omega} \}$ forms a frame for $\mathcal{X}$ and $(\{\Xi_{\omega}\}_{\omega} ,\{\Phi_{\omega}\}_{\omega})$ constitutes a continuous biframe for $\mathcal{X}$ with the biframe operator $\mathcal{T}_{\Xi,\Phi}$. By formula \ref{3}, we obtain $\mathcal{T}_{\Xi,\Phi} = \mathcal{T}_{\Phi,\Xi}$. For any $\xi \in \mathcal{X}$, we have
	\begin{equation*}
	\xi = \int_{\Omega}\langle \mathcal{T}^{-1}_{\Phi,\Xi} \xi, \Xi_{\omega} \rangle  \Phi_{\omega} \, d\mu(\omega) = \int_{\Omega} \langle \mathcal{T}^{-1}_{\Phi,\Xi} \xi, \Phi_{\omega} \rangle  \Xi_{\omega} \, d\mu(\omega).
	\end{equation*}
	Hence, $(\{\Xi_{\omega}\}_{\omega}$ and $\{\Phi_{\omega}\}_{\omega})$ are $G$-dual frames. We then have
	\begin{equation*}
	\Phi_{\omega} = (\mathcal{T}_{\Xi,\Phi}\mathcal{T}_{\Xi})^{-1}\Xi_{\omega} + e_{\omega} -\sum_{a=1}^{\infty} \langle \mathcal{T}^{-1}_{\Xi}\Xi_{\omega}, \Xi_{a} \rangle e_{a} = \mathcal{T}_{\Xi}\mathcal{T}_{\Xi,\Phi}^{-1}\Xi_{\omega} + e_{\omega} -\sum_{a=1}^{\infty} \langle \mathcal{T}^{-1}_{\Xi}\Xi_{\omega}, \Xi_{a} \rangle e_{a},
	\end{equation*}
	where $\{e_{\omega}\}$ is a Bessel sequence in $\mathcal{X}$.
	
	For $(ii)\Rightarrow (i)$, suppose there exists an operator $\mathcal{Q} \in B^{+}_{c,c}(\mathcal{X})$ such that
	\begin{equation*}
	\{\Phi_{\omega}\} = \{ (\mathcal{T}_{\Xi}\mathcal{Q})^{-1}\Xi_{\omega} + e_{\omega} - \sum_{a=1}^{\infty} \langle \mathcal{T}^{-1}_{\Xi}\Xi_{\omega}, \Xi_{a} \rangle e_{a} \},
	\end{equation*}
	where $\{e_{\omega}\}_{\omega\in \Omega}$ is a Bessel sequence in $\mathcal{X}$ and $\mathcal{Q} \in B^{+}_{c,c}(\mathcal{X})$.
	
	For any $\xi \in \mathcal{X}$, we have
	\begin{align*}
	\int_{\Omega} \langle \xi, \Xi_{\omega}\rangle \Phi_{\omega} \, d\mu(\omega) &= \int_{\Omega} \langle \xi, \Xi_{\omega}\rangle \left( (\mathcal{T}_{\Xi}\mathcal{Q})^{-1}\Xi_{\omega} + e_{\omega} - \sum_{a=1}^{\infty} \langle \mathcal{T}^{-1}_{\Xi}\Xi_{\omega}, \Xi_{a} \rangle e_{a} \right) \, d\mu(\omega) \\
	&= \int_{\Omega} \langle \xi, \Xi_{\omega}\rangle (\mathcal{T}_{\Xi}\mathcal{Q})^{-1}\Xi_{\omega} \, d\mu(\omega) + \int_{\Omega} \langle \xi, \Xi_{\omega}\rangle e_{\omega} \, d\mu(\omega) \\
	&\quad - \int_{\Omega} \sum_{a=1}^{\infty} \langle \xi, \Xi_{\omega}\rangle \langle \mathcal{T}^{-1}_{\Xi} \Xi_{\omega}, \Xi_{a}\rangle e_{a} \, d\mu(\omega) \\
	&= \mathcal{Q}^{-1} \xi + \int_{\Omega} \langle \xi, \Xi_{\omega}\rangle e_{\omega} \, d\mu(\omega) - \int_{\Omega} \sum_{a=1}^{\infty} \langle \xi, \Xi_{\omega}\rangle \langle \mathcal{T}^{-1}_{\Xi} \Xi_{\omega}, \Xi_{a}\rangle e_{a} \, d\mu(\omega) \\
	&= \mathcal{Q}^{-1} \xi + \int_{\Omega} \langle \xi, \Xi_{\omega}\rangle e_{\omega} \, d\mu(\omega) - \int_{\Omega} \langle \xi, \Xi_{\omega}\rangle e_{\omega} \, d\mu(\omega) \\
	&= \mathcal{Q}^{-1} \xi.
	\end{align*}
	This implies that $(\{\Xi_{\omega}\}_{\omega}, \{\Phi_{\omega}\}_{\omega})$ is a continuous biframe for $\mathcal{X}$. It is easy to check that $\{\Phi_{\omega}\}$ is a Bessel sequence, and $\xi = \int_{\Omega} \langle \mathcal{Q}\xi, \Xi_{\omega}\rangle \Phi_{\omega}$, making $\{\Phi_{\omega}\}_{\omega}$ a frame for $\mathcal{X}$.
	
	For $(iii) \Rightarrow (i)$, the result is obvious.
\end{proof}

\begin{theorem}
	Assume $\Xi = \{\Xi_{\omega}\}_{\omega\in \Omega}$ is a Riesz basis for the space $\mathcal{X}$, characterized by the frame operator $\mathcal{T}_{\Xi}$. The equivalence of the following statements is established:
	\begin{itemize}
		\item [(i)] The pair $(\{\Xi_{\omega}\}_{\omega\in \Omega}, \{\Phi_{\omega}\}_{\omega\in \Omega})$ serves as a continuous biframe for the space $\mathcal{X}$.
		\item [(ii)] The set $\{\Phi_{\omega}\}_{\omega\in \Omega}$ can be expressed as $\{ (\mathcal{T}_{\Xi}\mathcal{Q})^{-1} \Xi_{\omega} \}$, where $\mathcal{Q}$ belongs to $B^{+}_{b,b}(\mathcal{X})$.
	\end{itemize}
	Notably, satisfying either condition $(i)$ or $(ii)$ leads to the conclusion that $\{\Phi_{\omega}\}_{\omega\in \Omega}$ establishes a Riesz basis for the space $\mathcal{X}$.
	
\end{theorem}

\begin{proof}
	For $(i)\Rightarrow (ii)$, consider that $(\{\Xi_{\omega}\}_{\omega\in \Omega}, \{\Phi_{\omega}\}_{\omega\in \Omega})$ forms a continuous biframe for $\mathcal{X}$ with the frame operator $\mathcal{T}_{\Xi,\Phi}$. Given that $\{\Xi_{\omega}\}_{\omega\in \Omega}$ is a Riesz basis for $\mathcal{X}$, there exists an invertible operator $U \in B(\mathcal{X})$ such that $\Xi_{\omega} = Ue_{\omega}$ for all $\omega \in \Omega$, where $\{e_{\omega}\}_{\omega \in \Omega}$ is an orthonormal basis for $\mathcal{X}$.
	
	Referencing Theorem \ref{t3.10}, $\{\Phi_{\omega}\}_{\omega\in \Omega}$ is also a Riesz basis for $\mathcal{X}$. Therefore, we have
	\begin{align*}
	\Phi_{\omega} &= \mathcal{T}_{\Xi,\Phi}(U^{\ast})^{-1}U^{-1}e_{\omega} = \mathcal{T}_{\Xi,\Phi}U^{*-1} U^{-1} \Xi_{\omega} \\
	&= \mathcal{T}_{\Xi,\Phi}(UU^{*})^{-1} \Xi_{\omega} \\
	&= \mathcal{T}_{\Xi,\Phi} \mathcal{T}^{-1}_{\Xi}  \Xi_{\omega} \\
	&= (\mathcal{T}_{\Xi}\mathcal{T}^{-1}_{\Xi,\Phi})^{-1} \Xi_{\omega}.
	\end{align*}
	The proof for $(ii) \Rightarrow (i)$ follows similarly and can be easily derived.
	
\end{proof}

\begin{definition}
	Let $\{e_{\omega}\}_{\omega}$ be an orthonormal basis for $\mathcal{X}$.
	
	Define the set $[\{e_{\omega}\}] = \{\{\Xi_{\omega}\} \mid (\{e_{\omega}\}_{\omega},\{\Xi_{\omega}\}_{\omega})$ constitutes a continuous biframe for $\mathcal{X}\}$.
	
\end{definition}

\begin{proposition} \label{p2}
	Let $E = \{e_{\omega} \}_{\omega \in \Omega}$ be an orthonormal basis for the Hilbert space $\mathcal{X}$.
	
	The sequence $\Xi = \{\Xi_{\omega}\}_{\omega}$ is part of $[\{e_{\omega}\}]$ 
	if and only if there is an operator $U \in B^{+}_{b,b}(\mathcal{X})$ such that $\Xi_{\omega} = Ue_{\omega}$ for every $\omega \in \Omega$.
\end{proposition}

\begin{proof}
	Let $\Xi_{\omega} \in E$. The pair $(\{e_{\omega}\}_{\omega},\{\Xi_{\omega}\}_{\omega})$ is a continuous biframe with biframe operator $\mathcal{T}_{\Xi,G}$,
	\begin{equation*}
	\mathcal{T}_{\Xi,G}e_{k} = \int_{\Omega} \langle e_{k},e_{\omega}\rangle \Xi_{\omega}d\mu(\omega) = \Xi_{k}, \forall k \in \Omega.
	\end{equation*}
	Suppose there exists $U \in B^{+}_{b,b}(\mathcal{X})$ such that $\Xi_{\omega} = Ue_{\omega} \; \forall \omega \in \Omega$.
	
	For $\xi \in \mathcal{X}$,
	\begin{equation*}
	U\xi = \int_{\Omega} \langle \xi,e_{\omega}\rangle Ue_{\omega} d\mu(\omega) = \int_{\Omega} \langle \xi,e_{\omega}\rangle \Xi_{\omega}d\mu(\omega) = \mathcal{T}\xi.
	\end{equation*}
	This implies that $\mathcal{T} \in B^{+}_{b,b}(\mathcal{X})$ and $(\{e_{\omega}\}_{\omega},\{\Xi_{\omega}\}_{\omega})$ is a continuous biframe for $\mathcal{X}$ with biframe operator $\mathcal{T}$. Thus, $\{\Xi_{\omega}\}_{\omega} \in E$.
\end{proof}

\begin{proposition}
	Let $ \{e_{\omega} \}_{\omega \in \Omega} $ be an orthonormal basis for the Hilbert space $ \mathcal{X} $. Consider the sequence $ \{\Xi_{\omega}\}_{\omega} $ as a subset of $ \{e_{\omega}\} $. 
	If there exists another orthonormal basis $ \{\delta_{\omega}\}_{\omega} $ for $ \mathcal{X} $ such that  $ \Xi_{\omega} \in [ \{ \delta_{\omega} \}] $, then it must be that $ e_{\omega} = \delta_{\omega} $ for all $ \omega \in \Omega $.
\end{proposition}

	\begin{proof}
		According to Proposition \ref{p2}, there exist operators $U, V \in B^{+}_{b,b}(\mathcal{X})$ such that
		
		$ \Xi_{\omega} = U e_{\omega}$ and $ \Xi_{\omega} = V \delta_{\omega}$.
		We note that $ \Xi = \{\Xi_{\omega}\}_{\omega}$ is a frame for $ \mathcal{X}$ with a continuous biframe operator $ \mathcal{T}_{\Xi}$. For any element $ \xi \in \mathcal{X}$,
		\begin{equation*}
		\mathcal{T}_{\Xi}\xi = \int_{\Omega} \langle \xi, \Xi_{\omega} \rangle \Xi_{\omega} d\mu(\omega) = \int_{\Omega} \langle \xi, Ue_{\omega}\rangle Ue_{\omega} d\mu(\omega) = U^{2}\xi.
		\end{equation*}
		Similarly, $ \mathcal{T}_{\Xi}\xi = V^{2}\xi$ leads to the conclusion that $U^{2} = V^{2}$ and therefore $U = V$.
		
		Consequently, for each $ \omega \in \Omega$, it must be that $ e_{\omega} = \delta_{\omega}$.
	\end{proof}

We present a novel category of sequences.
\begin{definition}
	A sequence $ \{\Xi_{\omega}\}_{\omega} $ in $ \mathcal{X} $ is termed a continuous biframe-Riesz basis, or briefly b-Riesz, for $ \mathcal{X} $, if it is associated with an orthonormal basis $ \{e_{\omega}\}_{\omega} $ for $ \mathcal{X} $. This association is such that the pair $ \{\{e_{\omega}\}_{\omega}, \{\Xi_{\omega}\}_{\omega}\} $ forms a continuous biframe for $ \mathcal{X} $.
\end{definition}

\begin{theorem} \label{t21}
	Let $ \{\Xi_{\omega}\}_{\omega} $ be a sequence in $ \mathcal{X} $. Then the following statements are equivalent:
	\begin{itemize}
		\item [(i)] $ \{\Xi_{\omega}\}_{\omega} $ is a continuous biframe-Riesz basis.
		\item [(ii)] $ \{\Xi_{\omega}\}_{\omega} $ is a member of $ E $ where $ E=\{e_{\omega}\}_{\omega} $ is an orthonormal basis.
		\item [(iii)] There exists an orthonormal basis $ \{e_{\omega}\}_{\omega} $ and an operator $ U \in B^{+}_{b,b}(\mathcal{X}) $ such that $ \{\Xi_{\omega} \}_{\omega} = \{Ue_{\omega}\}_{\omega} $.
		\item[(iv)] $ \{\Xi_{\omega}\} _{\omega} \in E $ for $ E $  an orthonormal basis.
	\end{itemize}
\end{theorem}

		Drawing from Proposition \ref{p2}, we infer that every orthonormal basis qualifies as both a biframe-Riesz and a Riesz basis.
	\begin{example}
		Consider the Riesz basis $ \{\Xi_{\omega}\}_{\omega} = \{(1,-2), (-1,0)\} $ for $ \mathbb{R}^{2} $ and an orthonormal basis $ \{e_{\omega}\}_{\omega} = \left\{\left(\frac{1}{2}, \frac{\sqrt{3}}{2}\right), \left(-\frac{\sqrt{3}}{2}, \frac{1}{2}\right)\right\} $ for $ (x,y) \in \mathbb{R}^{2} $,
		\begin{equation*}
		\langle (x,y),\left(\frac{1}{2}, \frac{\sqrt{3}}{2}\right)\rangle \langle (1,-2), (x,y)\rangle + \langle (x,y),\left(\frac{1}{2}, \frac{\sqrt{3}}{2}\right)\rangle \langle (-1,0), (x,y)\rangle = -xy - \sqrt{3}y^{2}.
		\end{equation*}
		We observe that $ (\{e_{\omega}\}, \{\Xi_{\omega}\}) $ forms a biframe for $ \mathbb{R}^{2} $.
	\end{example}
	
\begin{proposition}
	Let $ \{e_{\omega}\}_{\omega} $ be an orthonormal basis for $ \mathcal{X} $, and suppose $ \{\Xi_{\omega}\}_{\omega} = \{Ue_{\omega}\}_{\omega} $ and $ \{\Phi_{\omega}\}_{\omega} = \{Ve_{\omega}\}_{\omega} $ are elements of $ [\{e_{\omega}\}] $. If the operator $ VU $ is positive, then $ (\{\Xi_{\omega}\}_{\omega}, \{\Phi_{\omega}\}_{\omega}) $ constitutes a continuous biframe for $ \mathcal{X} $.
\end{proposition}

\begin{proof}
	Referencing Theorem \ref{t21}, we have operators $ U, V \in B^{+}_{b,b}(\mathcal{X}) $,
	for any element $ \xi \in \mathcal{X} $,
	\begin{equation*}
	VU\xi = VU \int_{\Omega} \langle \xi, e_{\omega}\rangle e_{\omega} d\mu(\omega) = \int_{\Omega} \langle \xi, Ue_{\omega} \rangle Ve_{\omega} d\mu(\omega) = \int_{\Omega} \langle \xi, \Xi_{\omega}\rangle \Phi_{\omega} d\mu(\omega) = \mathcal{T}\xi.    
	\end{equation*}
	If $ VU $ is positive, then $ \mathcal{T} \in B^{+}_{b,b}(\mathcal{X}) $ and thus $ (\{\Xi_{\omega}\}_{\omega}, \{\Phi_{\omega}\}_{\omega}) $ forms a continuous biframe for $ \mathcal{X} $.
\end{proof}

We explore the characteristics of a continuous biframe where one of the component sequences is a continuous biframe-Riesz basis for $ \mathcal{X} $.

\begin{theorem}
	Suppose $ \{\Xi_{\omega}\}_{\omega} $ is a continuous biframe-Riesz basis for $ \mathcal{X} $, and $ \{\Phi_{\omega}\}_{\omega} $ is another sequence in $ \mathcal{X} $. The following conditions are then equivalent:
	\begin{itemize}
		\item [(i)] The pair $ (\{\Xi_{\omega}\}_{\omega}, \{\Phi_{\omega}\}_{\omega}) $ constitutes a continuous biframe for $ \mathcal{X} $.
		\item [(ii)] There exist operators $ U, V \in B^{+}_{b,b}(\mathcal{X}) $ such that $ \Xi_{\omega} = Ue_{\omega} $ and $ \Phi_{\omega} = VU^{-1}e_{\omega} $, where $ \{e_{\omega}\}_{\omega} $ represents an orthonormal basis for $ \mathcal{X} $ for every $ \omega \in \Omega $.
	\end{itemize}
\end{theorem}

	\begin{proof}
		For the implication $ (i) \Rightarrow (ii) $:
		
		Assume that $ (\{\Xi_{\omega}\}_{\omega}, \{\Phi_{\omega}\}_{\omega}) $ is a continuous biframe with biframe operator $ \mathcal{T}_{\Xi,\Phi} \in B^{+}_{b,b}(\mathcal{X}) $. Given that $ \{\Xi_{\omega}\}_{\omega}$ is a continuous biframe-Riesz basis, and by Theorem \ref{t21}, there exists an orthonormal basis $ \{e_{\omega}\}_{\omega} $ for $\mathcal{X} $ and an operator $ U \in B^{+}_{b,b}(\mathcal{X})$ such that $ \Xi_{\omega} = Ue_{\omega} $, for all $\omega \in \Omega $.
		
		For $ \xi \in \mathcal{X} $,
		\begin{equation*}
		\mathcal{T}_{\Xi,\Phi}\xi = \int_{\Omega}\langle \xi, \Xi_{\omega}\rangle \Phi_{\omega}d\mu(\omega) = \int_{\Omega}\langle \xi, Ue_{\omega}\rangle \Phi_{\omega}d\mu(\omega) = \int_{\Omega}\langle U\xi, e_{\omega}\rangle \Phi_{\omega}d\mu(\omega),
		\end{equation*}
		and for   $ k \in \Omega $,
		\begin{equation*}
		\mathcal{T}_{\Xi,\Phi}U^{-1}e_{k} = \int_{\Omega}\langle UU^{-1}e_{k}, e_{\omega}\rangle \Phi_{\omega}d\mu(\omega) = \int_{\Omega}\langle e_{k}, e_{\omega}\rangle \Phi_{\omega}d\mu(\omega) = \Phi_{k}.
		\end{equation*}
		
		Now, assume that $ (ii) $ is true.
		
		For $ \xi \in \mathcal{X} $,
		\begin{equation*}
		\int_{\Omega}\langle \xi, \Xi_{\omega} \rangle \Phi_{\omega}d\mu(\omega) = \int_{\Omega} \langle \xi, Ue_{\omega} \rangle VU^{-1}e_{\omega}d\mu(\omega) = VU^{-1}\int_{\Omega} \langle U\xi, e_{\omega} \rangle e_{\omega}d\mu(\omega) = VU^{-1}U\xi = V\xi.
		\end{equation*}
		We conclude that $ \mathcal{T}_{\Xi,\Phi} = V \in B^{+}_{b,b}(\mathcal{X}) $, and by Theorem \ref{t8}, $ (\{\Xi_{\omega}\}_{\omega}, \{\Phi_{\omega}\}_{\omega}) $ is a continuous biframe.
	\end{proof}
	
	\begin{corollary} \label{c26}
		Suppose $ \{\Xi_{\omega}\}_{\omega} $ is a continuous biframe-Riesz basis for $ \mathcal{X} $ and $ \{\Phi_{\omega}\}_{\omega} $ is a sequence in $ \mathcal{X} $. Then the following statements are equivalent:
		\begin{itemize}
			\item [(i)] $ (\{\Xi_{\omega}\}_{\omega}, \{\Phi_{\omega}\}_{\omega}) $ forms a continuous Parseval biframe for $ \mathcal{X} $.
			\item [(ii)] There exists $ U \in B^{+}_{b,b}(\mathcal{X}) $ such that $ \Xi_{\omega} = Ue_{\omega} $ and $ \Phi_{\omega} = U^{-1}e_{\omega} $, where $ \{e_{\omega}\}_{\omega} $ is an orthonormal basis for $ \mathcal{X} $ applicable to every $ \omega \in \Omega $.
		\end{itemize}
	\end{corollary}
	
	\begin{corollary}
		If $ \{\Xi_{\omega}\}_{\omega} $ is a continuous biframe-Riesz basis for $ \mathcal{X} $, then both the canonical dual of $ \{\Xi_{\omega}\}_{\omega} $ and any biorthogonal frame to $ \{\Xi_{\omega}\}_{\omega} $ also qualify as continuous biframe-Riesz bases for $ \mathcal{X} $.
	\end{corollary}

	\begin{proof}
		If $ \{\Xi_{\omega}\}_{\omega} $ is a Riesz basis, its sole dual frame is the canonical dual frame.
	\end{proof}
	
	Following Corollaries \ref{c13} and \ref{c26}, it is established that every dual of $ \{\Xi_{\omega}\}_{\omega} $ and every biorthogonal frame to $ \{\Xi_{\omega}\}_{\omega} $ constitute a continuous Parseval biframe with $ \{\Xi_{\omega}\}_{\omega} $. These frames are continuous biframe-Riesz bases for $ \mathcal{X} $.
	
	\section*{Declarations}
	
	\medskip
	
	\noindent \textbf{Availablity of data and materials}\newline
	\noindent Not applicable.
	
	\medskip
	
	\noindent \textbf{Human and animal rights}\newline
	\noindent We would like to mention that this article does not contain any studies
	with animals and does not involve any studies over human being.

	\medskip

	\noindent \textbf{Conflict of interest}\newline
	\noindent The authors declare that they have no competing interests.

	\medskip
	
	\noindent \textbf{Fundings} \newline
	\noindent The authors declare that there is no funding available for this paper.

	\medskip
	
	\noindent \textbf{Authors' contributions}\newline
	\noindent The authors equally conceived of the study, participated in its
	design and coordination, drafted the manuscript, participated in the
	sequence alignment, and read and approved the final manuscript. 
	
	\medskip
	
	\bibliographystyle{amsplain}
	
\end{document}